\documentclass{article}

\usepackage[backref]{hyperref}
\usepackage{amsrefs,amscd}

\usepackage{amssymb,amsmath,amsthm}
\usepackage{latexsym}
\usepackage[all]{xy}
\usepackage{xcolor}

\newcommand{\noproof}{\hspace{\stretch{1}} $\Box$}

\theoremstyle{plain}
\newtheorem{theorem}{Theorem}[section]
\newtheorem{proposition}[theorem]{Proposition}
\newtheorem{lemma}[theorem]{Lemma}

\theoremstyle{definition}
\newtheorem{definition}[theorem]{Definition}
\newtheorem{example}[theorem]{Example}
\newtheorem{remark}[theorem]{Remark}

\DeclareMathOperator{\St}{Sk}
\DeclareMathOperator{\diam}{diam}
\DeclareMathOperator{\id}{id}

\newcommand{\reals}{\mathbb R}
\newcommand{\R}{\mathbb R}

\newcommand{\naturals}{\mathbb{N}}

\newcommand{\abs}[1]{\left|#1\right|}
\newcommand{\boundary}{\partial}
\newcommand{\coarse}{\mathrm{coarse}}

\newcounter{commentcounter}

\usepackage{ifthen} 
\newcommand{\showcomments}{yes}

\newsavebox{\commentbox}
\newenvironment{com}%
{\ifthenelse{\equal{\showcomments}{yes}}%
{\footnotemark
        \begin{lrbox}{\commentbox}
        \begin{minipage}[t]{0.93in}\raggedright\sffamily\tiny
        \footnotemark[\arabic{footnote}]}
{\begin{lrbox}{\commentbox}}}
{\ifthenelse{\equal{\showcomments}{yes}}
{\end{minipage}\end{lrbox}\marginpar{\usebox{\commentbox}}}
{\end{lrbox}}}

\begin{document}

\title{Coarse Homotopy Groups}

\author{Paul
  D. Mitchener\thanks{\protect\href{mailto:P.Mitchener@shef.ac.uk}{e-mail:P.Mitchener@shef.ac.uk}\protect\\
  \protect\href{http://www.mitchener.staff.shef.ac.uk/}{www:~http://www.mitchener.staff.shef.ac.uk/}}\\
School of Mathematics\\
The University of Sheffield\\ UK \and Behnam Norouzizadeh\\ Teradici\\
Vancouver\\ Canada \and Thomas Schick\thanks{
\protect\href{mailto:thomas.schick@math.uni-goettingen.de}{e-mail:
  thomas.schick@math.uni-goettingen.de}
\protect\\
\protect\href{http://www.uni-math.gwdg.de/schick}{www:~http://www.uni-math.gwdg.de/schick}}\\
Mathematisches Institut\\
Universit\"at G{\"o}ttingen\\ Germany}

\date{}
\maketitle

\begin{abstract}
  In this note on coarse geometry we revisit coarse homotopy. We prove that
  coarse homotopy indeed is an equivalence relation, and this in the most
  general context of abstract coarse structures. We introduce (in a geometric
  way) coarse homotopy groups. The main result is that the coarse homotopy
  groups of a cone over a compact simplicial complex coincide with the usual
  homotopy groups of the underlying compact simplicial complex.

  To prove this we develop geometric triangulation techniques for cones which
  we expect to be of relevance also in different contexts.
\end{abstract}

\maketitle


\section{Introduction and Preliminaries}

Our main results are the definition and computation of coarse homotopy and
coarse 
homotopy 
groups, in the category of generalized coarse spaces, as introduced in
particular by John Roe \cite{Roe6}.

In this note, we discuss in detail the concept of coarse homotopy (and
coarse homotopy equivalence). In particular, we check carefully that this is an
equivalence 
relation, a result which seems not to be available in the literature. We use
the ``correct'' notion of coarse homotopy, differing from the original one
which has been shown to be inappropriate by being too flexible.

We then introduce a geometric version of coarse homotopy groups and show their
basic properties (in particular that they form groups in the first place). The
main computation is then the calculation of the coarse homotopy groups of
cones on simplicial complexes: they are equal to the homotopy groups of the
base of the cone. Preliminary results
in this direction are contained in the G\"ottingen doctoral thesis of Behnam
Norouzizadeh \cite{Norouzizadeh}.

Along the way, we discuss that there is a canonical coarse
structure on the (euclidean) cone of a simplicial complex. We also develop
precise geometric triangulation techniques for cones of simplicial complexes
which we expect to be of relevance in other contexts. 

\smallskip
Before we get to these results, we start with preliminaries, introducing the
coarse category and then deriving some gluing theorems for coarse maps, which
are indispensable when working with geometric homotopy groups. We also work
out some basics of triangulations and subdivisions which we will need.

\subsection{The coarse category}

Recall (compare e.g.~\cite{Roe6}) that a (unital) {\em coarse structure} on a set $X$ is a distinguished collection, $\mathcal E$, of subsets of the product $X\times X$ called {\em entourages} such that:

\begin{itemize}

\item Any finite union of entourages is an entourage.  Any subset of an entourage is an entourage.

\item The union of all entourages is the entire space $X\times X$.

\item The {\em inverse} of an entourage $M$:
\[ M^{-1} = \{ (y,x)\in X\times X \ |\ (x,y)\in M \} \]
is an entourage.

\item The {\em composition} of entourages $M_1$ and $M_2$:
\[ M_1 M_2 = \{ (x,z)\in X\times X \ |\ (x,y)\in M_1 \ (y,z)\in M_2 \textrm{ for some }y\in X \} \]
is an entourage.

\item The diagonal, $\Delta = \{ (x,x) \ |\ x\in X \}$ is an entourage.

\end{itemize}

A space $X$ equipped with a coarse structure is called a {\em coarse space}.

The above definition differs slightly from that of \cite{HPR}, but agrees with
the definition of a {\em unital coarse structure} on a set in \cite{Mitch4,
  STY}.  We call an entourage {\em symmetric} if it is equal to its inverse
and we write $S(M):=M\cup M^{-1}$ for the symmetric entourage generated by the
entourage $M$.

If $X$ is a coarse space, and $f,g\colon S\rightarrow X$ are maps into $X$, the maps $f$ and $g$ are termed {\em close} or {\em coarsely equivalent} if the set $\{ (f(s) , g(s)) \ |\ s\in S \}$
is an entourage.  We call a subset $B\subseteq X$ {\em bounded} if
the inclusion $B\hookrightarrow X$ is close to a constant map.


The most important, and motivating, example of a coarse structure is the one
of a proper metric space.
\begin{example}
Let $X$ be a proper metric space (i.e.~the closures of sets of finite diameter
are compact). The {\em bounded coarse structure}
on $X$ is by definition the unital coarse structure formed by defining the
entourages to be subsets of {\em $R$-neighbourhoods of the diagonal}:
\[ D_R = \{ (x,y) \in X\times X \ |\ d(x,y)< R \} ;\qquad R\in\reals. \]

The bounded sets are simply those which are bounded with respect to the metric.
\end{example}




Let $X$ and $Y$ be coarse spaces.  Then a map $f\colon X\rightarrow Y$ is said to be {\em controlled} if for every entourage $M\subseteq X\times X$, the image$$f[M] = \{ (f(x),f(y))\ |\ (x,y)\in M \}$$
is an entourage.  A controlled map is called {\em coarse} if the inverse image of a bounded set is also bounded.

If $X$ and $Y$ are metric spaces equipped with their bounded coarse structures, a map $f\colon X\rightarrow Y$ is controlled if and only if for all $R>0$ there exists $S>0$ such that if $d(x,y)<R$ for $x,y\in X$, then $d(f(x),f(y))<S$ in the space $Y$.


We can form the category of all coarse spaces and coarse maps.  We call this category the {\em coarse category}.  We call a coarse map $f\colon X\rightarrow Y$ a {\em coarse equivalence} if there is a coarse map $g\colon Y\rightarrow X$ such that the composites $g\circ f$ and $f\circ g$ are close to the identities $1_X$ and $1_Y$ respectively.

Coarse spaces $X$ and $Y$ are said to be {\em coarsely equivalent} if there is a coarse equivalence between them.  There is a similar notion of coarse equivalence between pairs of coarse spaces.

The following definition comes from \cite{HPR}.

\begin{definition}
Let $X$ be a Hausdorff space.  A coarse structure on $X$ is said to be {\em compatible with the topology} if every entourage is contained in an open entourage, and the closure of any bounded set is compact.
\end{definition}

Note that any coarse topological space is locally compact.  In such a space, the bounded sets are {\em precisely} those which are precompact.  It also follows from the definition that any precompact subset of the product of the space with itself is an entourage, and the closure of any entourage is an entourage.

\begin{example}
The bounded coarse structure on a proper metric space is compatible with the topology.  
\end{example}

The purpose of this article is to develop some notions of homotopy theory in
the coarse category. These homotopies have to end eventually, but the end time
will be allowed to depend on the given point in the coarse space (and to go to
infinity as one goes to infinity). This will be measured by coarse maps
$p\colon X\rightarrow \R_+$, which we call ``basepoint projection'', and which
will be part of the structure for us.

\begin{example} \label{Dranishnikov}
Let $X$ be a proper metric space.  Endow $\R_+$ with the bounded coarse structure coming from the metric. Choose a point $x_0
\in X$. Then we have a basepoint projection
$p_{x_0}\colon X\rightarrow \R_+$ defined by the formula
\[ p_{x_0} (x) = d(x,x_0) . \]

Observe that for any two points $x_0 , y_0 \in X$ the maps $p_{x_0}$ and $p_{y_0}$ are close.
\end{example}

When proving results about coarse homotopies, the following lemma summarises many of the relevant properties of the coarse space $\R_+$.  It is easy to check.

\begin{lemma} \label{GR}
Let $\R_+$ be the space $[0,\infty )$ equipped with the bounded coarse structure arising from the usual metric.  Then the following hold:

\begin{itemize}

\item Let $M,N\subseteq \R_+ \times \R_+$ be entourages.  Then the sets
\[ M+N = \{ (u+x, v+y)\ |\ (u,v)\in M, (x,y)\in N \} \]
and
\[ M-N = \{ (u-x, v-y)\ |\ (u,v)\in M, (x,y)\in N, u\geq x, v\geq y \} \]
are entourages.

\item Let $M\subseteq \R_+ \times \R_+$ be an entourage.  Then the set
\[ Z(M) = \{ (u,v)\in \R_+\times \R_+\ |\ x\leq u\leq y,\ x\leq v\leq y,\ (x,y)\in M \} \]
is an entourage.  Note that $Z(Z(M)) = Z(M)$.

\item Let $M\subseteq \R_+ \times \R_+$ be an entourage.  Then the set
\[ \{ (x+a,y+a)\ |\ a\in \R_+, (x,y)\in M \} \]
is an entourage.

\end{itemize}

\noproof
\end{lemma}


\begin{proposition} \label{sumax}
Let $X$ be a coarse space, and let $f,g\colon X\rightarrow \R_+$ be coarse maps.  Then the sum of $f$ and $g$ and the maximum of $f$ and $g$ are coarse maps.
\end{proposition}

\begin{proof}
Let $M\subseteq X\times X$ be an entourage.  The images $f[M]$
and $g[N]$ are entourages.  Observe that
\[ (f+g)[M] = \{ (f(x)+g(x),f(y)+g(y)) \ | \ (x,y)\in M \} \subseteq
f[M]+g[N] \] and
\begin{equation*}
  \begin{split}
    \max (f,g)[M]& = \{ (\max\{f(x),g(x)\},\max\{f(y),g(y)\}) \ | (x,y)\in M\}\\
    &\subseteq f[M]\cup g[M] \cup S(Z(S(f(M)))\cup Z(S(g(M)))).
  \end{split}
\end{equation*}

Hence by the above, the images $(f+g)[M]$ and
$\max (f,g)[M]$ are entourages.

Now, let $B\subseteq \R_+$ be bounded.  Then we can choose $a>0$ such that $B\subseteq [0,a]$.  Hence
\[ (f+g)^{-1}[B] \subseteq \{ x\in X \ | \ f(x)+ g(x)\leq a \}
\subseteq \{ x\in X \ | \ f(x)\leq a \} = f^{-1}[0,a] \]

We see that the inverse image $(f+g)^{-1}[B]$ is bounded.  A similar
argument tells us that the inverse image $\max (f,g)^{-1}[B]$ is
bounded.

So the maps $f+g$ and $\max (f,g)$ are both coarse, and we are done.
\end{proof}

\begin{definition}
Let $X$ and $Y$ be coarse spaces.  Then the set-theoretic product $X\times Y$ is equipped with the coarse structure defined by taking the entourages to be subsets of sets of the form $M\times N$, where $M\subseteq X\times X$ and $N\subseteq Y\times Y$ are entourages for the spaces $X$ and $Y$ respectively.
\end{definition}

The product $X\times Y$ is {\em not} a product in the category-theoretic sense.  The problem is that the projections $\pi_X\colon X\times Y\rightarrow X$ and $\pi_Y\colon X\times Y\rightarrow X$ are not in general coarse maps; the inverse images of bounded sets need not to be bounded.


\subsection{Pasting together maps}

Many of the constructions of maps we are going to make are carried out in a
piecewise manner, and we need criteria which make sure that a map which has
good properties on the pieces does have such good properties globally.

For the following, recall that for metric spaces $X$ and $Y$ we call a map
$f\colon X\rightarrow Y$ {\em Lipschitz} if there is a constant $C>0$ such
that $d(f(x),f(y))\leq Cd(x,y)$ for all $x,y\in X$.  Certainly, any Lipschitz
map is continuous.  We call $C$ the {\em Lipschitz constant} of $f$; a Lipschitz map with Lipschitz constant $C$ is called {\em $C$-Lipschitz}.
A {\em bilipschitz homeomorphism} is an invertible Lipschitz map with Lipschitz inverse.

We will need the following properties of Lipschitz maps which are well known
and easy to prove.

\begin{lemma}\label{lem:basic_lipschitz}
  A continuous and piecewise smooth map between smooth Riemannian manifolds
  with a 
  uniform bound on the norm of the differential is Lipschitz.

  A composition of Lipschitz maps is Lipschitz.
\noproof
\end{lemma}

\begin{lemma}\label{lem:Lip_extend}
Let $X$ be a geodesic metric space with a decomposition $X=A\cup B$ for closed subsets $A$ and $B$. Let $Y$ be a metric space,  and let $f\colon X\to Y$ be
a map such that the restrictions $f|_A\colon
A\to Y$ and $f|_B\colon B\to Y$ are both $C$-Lipschitz. Then also
$f\colon X \to Y$ is $C$-Lipschitz.

More generally, if $X=\bigcup_{i\in I} A_i$ is a union of
closed subsets $A_i$, which is such that every compact subset is contained in a union of only finitely many of the $A_i$, and the restriction $f|_{A_i}\colon
A_i\to Y$ is $C$-Lipschitz for every $i\in I$, then the map $f\colon X\to Y$ is also $C$-Lipschitz.
\end{lemma}

\begin{proof}
  Pick $x,y\in X$. If $x,y\in A$ or $x,y\in B$, then $d(f(x),f(y))\le C d(x,y)$ by the
Lipschitz condition on $f|_A$. and similar if $x,y\in B$. If $x\in A$
and $y\in B$ choose a geodesic $\gamma\colon [0,d(x,y)]\to X$ from $x$
to $y$. Then, there is a point $z\in A\cap B$ on that geodesic. We obtain
$$ d(f(x),f(y))\le d(f(x),f(y)) + d(f(y),f(y))\le C d(x,z)+Cd(z,y) =
Cd(x,y).$$

Here, the first inequality is the triangle inequality, the second the
Lipschitz property of $f|_A$ and $f|_B$ and the third the geodesic
property of $\gamma$.

The same proof gives the general statement for $X=\bigcup_{i\in I} A_i$, using
the fact that any geodesic has compact 
image and therefore will involve only finitely many of the $A_i$.
\end{proof}

\begin{proposition}\label{prop:coarse_glue}
  Let $X$ be a proper metric space, considered as coarse space. Assume
  $X=A\cup B$. Assume this decomposition is \emph{coarsely excisive}, i.e.~for
  each $R>0$ there is $S>0$ such that $U_R(A)\cap U_R(B)\subset U_S(A\cap B)$,
  where $U_R(Z):=\{x\in X\mid d(x,Z)\le R\}$ for $Z\subset X$ is the
  $R$-neighborhood of $Z$.

  Assume $f\colon X\to Y$ for a coarse space $Y$ satisfies that $f|_A\colon
  A\to Y$ and $f|_B\colon B\to Y$ are coarse. Then also $f\colon X\to Y$ is
  coarse. 
\end{proposition}

\begin{proof}
  Firstly, if $K\subset Y$ is bounded then $f^{-1}(Y)=(f|_A)^{-1}(K)\cup
  (f|_B)^{-1}(K)$ is the union of two bounded sets and therefore bounded.

  Secondly, given $R>0$ choose $S>0$ such
  that $U_R(A)\cap U_R(B)\subset U_S(A\cap B)$. We have to show that the
  $R$-entourage $\{(x,y)\in X\times X\mid d(x,y)\le R\}$ in $X\times X$ is
  mapped to an entourage of $Y$. 

Let $A_R= \{ (x,y)\in A\times A \ |\ d(x,y)\le R \}$ and $B_R = \{ (x,y)\in B\times B \ |\ d(x,y)\le R \}$. Then the $R$-entourage $\{(x,y)\in X\times X\mid d(x,y)\le R\}$ is the union of the sets $A_R$, $B_R$, and the set
  $C:=\{(x,y)\in A\times B\cup B\times A\mid d(x,y)\le R\}$. Hence,
  if $(x,y)\in C$, then $x,y\in U_R(A)\cap U_R(B)\subset U_S(A\cap B)$,
  i.e.~the set $C$ is contained in the $S$-entourage of $A\cap B$ and therefore also
  in the $S$-entourage of $A$.

  As the restrictions $f|_A$ and $f|_B$ are coarse maps, the above considerations imply that  the
  images of $A_R$, $B_R$, and $C$ are each entourages. Consequently, the map $f$ is
  coarse. 
\end{proof}

\subsection{Simplicial complexes}

The following lemma summarizes metric properties of simplicial maps
between geometric simplices, known by elementary geometry.
\begin{lemma}\label{lem:simplicial_is_lipschitz}
  Let $\sigma:=\langle v_0,\dots,v_n\rangle \subset \reals^N$ be a geometric
  $n$-simplex in $\reals^N$ spanned by $(n+1)$ vectors
  $v_0,\dots,v_n$ in general position. Let $w_0,\dots,w_n$ be vertices of a geometric $k$-simplex
  $\tau\subset\reals^M$.

  Then there is a unique affine linear map $f\colon \sigma\to \tau$ sending
  $v_i$ to $w_i$. The Lipschitz constant of $f$ is bounded above by
  $c(n,k,w){\max\{|w_i-w_j|\}}$
  where $c(n,k,w)$ depends on the dimensions $n$ and $k$ of the simplices and
  in addition on a lower bound $w$ on the width of $\sigma$ defined to be the
  shortest distance from any vertex of $\sigma$ to the opposite face.
\noproof
\end{lemma}

We need specific, geometric, triangulations of $c(X)$ for a finite simplicial
complex $X$ embedded simplicially into $\reals^n$. This can be achieved using
standard subdivisions, as introduced by Whitney \cite{Whitney} and used by
Dodziuk \cite[Section 2]{Dodziuk}.

\begin{definition}
  A simplicial complex $X$ is called \emph{locally ordered} if there is a partial
  ordering on its vertices which restricts to a total ordering on the vertices
  of each simplex of $X$.
\end{definition}
\begin{example}
  A total order on the vertices of a simplicial complex of course also is a
  partial order. A barycentric subdivision has a canonical local order.
\end{example}

\begin{definition}
  Let $\sigma:=\langle v_0,\dots,v_n\rangle \subset \reals^N$ be a simplex
  realized as convex hull of the $n+1$ affinely independent vertices
  $v_0,\dots,v_n\in\reals^N$. We define its standard subdivision $S(\sigma)$
  as the 
  simplicial complex with vertices $v_{ij}:=(v_i+v_j)/2$ for $0\le i\le j\le
  n$. 

    On this set of vertices we define a partial order setting $(i,j)\le (k,l)$
    if and only if $k\le i\le j\le l$. By definition, the simplices of the
    standard subdivision are spanned by increasing sequences of vertices,
    making $S(\sigma)$ locally ordered.

    Given a locally ordered simplicial complex $X$ define a standard
    subdivision $S(X)$ by applying the standard
    decomposition to each simplex to obtain a simplicial decomposition of the
    whole simplicial complex. This is well defined due to the compatibility of
    the local orders of the vertices of the different simplicies. Note that
    the vertices of the standard subdivision  
    inherit a partial order making it locally ordered which allows us to iterate the standard subdivision
    procedure. 
  \end{definition}

  \begin{definition}
   Two geometric simplices $\sigma,\tau\subset\reals^N$ are \emph{strongly 
    similar} if one can be obtained from the other by translation and
    multiplication by a positive constant.
  \end{definition}
  
  The great advantage of the standard subdivision is \cite[Lemma 2.5]{Dodziuk}:
  \begin{lemma}\label{lem:Dodziuk}
    Let $X$ be a finite simplicial complex embedded into $\reals^N$, with a
    local order. There are
    only finitely many strong similarity types of the simplices of iterated
    standard subdivisions of $X$.
  \end{lemma}

  We will also need several simplicial structures on $X\times [0,1]$ for a
  simplicial complex $X$.

  \begin{definition}\label{def:product_simpl}
    Recall that, for a locally ordered simplicial complex $X$ there is a
    canonical 
    triangulation of $X\times [0,1]$ with the obvious simplicies in $X\times
    \{0\}$ and $X\times \{1\}$ coming from the triangulation of $X$ and where in
    addition for any ordered simplex 
    $(v_0,\dots,v_k)$ of $X$ and $0\le j\le k$ we get a new simplex spanned by
    $(v_0,0),\dots (v_j,0),(v_j,1),\dots, (v_k,1)$.

    We now define a ``standard product subdivision''
    which restricts to the given triangulation on $X\times \{0\}$ but to the
    standard subdivision $S(X)\times \{1\}$ on the other end.
    The additional simplices here are the following:

    whenever $u_l\le u_{l-1}\le\dots u_0\le v_0<\dots <v_k\le w_0\le\dots w_l$
    are vertices of a simplex of $X$ such that
    $(u_0,w_0)<(u_1,w_1)<\dots (u_l,w_l)$ in the standard subdivision 
    we get a simplex of the ``standard product subdivision'' of  $X\times
    [0,1]$ spanned by $(v_0,0),\dots, (v_k,0),
    ((u_0,w_0),1), \dots, ((u_l,w_l),1)$. It is a little combinatorial
    exercise that these simplices are indeed precisely and in unique way
    unions of the simplices of the canonical triangulation of $S(X)\times
    [0,1]$ (which therefore further refines our standard product subdivision):
    the convex hull of $\langle(u_0,w_0), \dots, (u_l,w_l)\rangle \times\{1\}$
    and 
    $\langle v_0,\dots, v_k\rangle \times \{0\}$ as above is precisely the
    union of the convex hulls of 
    $\langle(u_0,w_0), \dots, (u_l,w_l)\rangle\times\{1\}$ and the simplices
    in the standard subdivision of $\langle v_0,\dots,v_k\rangle$ (times
    $\{0\})$, and this way we obtain precisely the simplices in the canonical
    triangulation of $S(X)\times [0,1]$.

    Therefore the described standard product subdivision indeed giving a
    triangulation of $X\times [0,1]$.
  \end{definition}

\section{Coarse Homotopy}

To develop the notion of homotopy for coarse spaces we first consider
cylinders. Our definition is inspired by \cite[Section 3]{Dran}.

\begin{definition}
Let $X$ be a coarse space, and let $p\colon X\rightarrow \R_+$ be a coarse map.  Then we define the {\em $p$-cylinder}
\[ I_p X = \{ (x,t)\in X\times \R_+ \ |\ t\leq p(x)+1 \}. \]
\end{definition}

We have inclusions $i_0\colon X\rightarrow I_p X$ and $i_1\colon
X\rightarrow I_p X$ defined by the formulas $i_0(x) = (x,0)$ and
$i(x) = (x,p(x)+1)$, respectively.  The canonical projection $q\colon I_p
X\rightarrow X$ defined by the formula $q(x,t)=x$ is a coarse map.  The identities
$q\circ i_0 =1_X$ and $q\circ i_1 = 1_X$ clearly hold.

\begin{definition}
Let $X$ and $Y$ be coarse spaces.  A {\em coarse homotopy}
is a coarse map $H\colon I_p X\rightarrow Y$ for some coarse map $p\colon X\rightarrow \R_+$.

We call coarse maps $f_0 \colon X\rightarrow Y$ and $f_1 \colon X\rightarrow Y$ {\em
coarsely homotopic} if there is a coarse homotopy $H\colon
I_p X\rightarrow Y$ such that $f_0 = H\circ i_0$ and $f_1 = H\circ
i_1$.

This map $H$ is termed a \emph{coarse homotopy} between the maps $f_0$ and
$f_1$.
\end{definition}

Let $f\colon X\rightarrow Y$ be a coarse map between coarse
spaces.  We call the map $f$ a {\em coarse homotopy
equivalence} if there is a coarse map $g\colon Y\rightarrow X$
such that the composites $g\circ f$ and $f\circ g$ are coarsely
homotopic to the identities $1_X$ and $1_Y$ respectively.

\begin{example} \label{close}
Let $X$ and $Y$ be coarse spaces.  Let $p\colon X\rightarrow \R_+$ be any coarse map.  Let $f_0\colon X\rightarrow
Y$ and $f_1\colon X\rightarrow Y$ be close coarse maps.  Then we can
define a coarse homotopy $H\colon I_pX\rightarrow Y$ between the
maps $f_0$ and $f_1$ by the formula
\[ H(x,t) = \left\{ \begin{array}{ll}
f_0 (x) & t< 1 \\
f_1 (x) & t\geq 1 \\
\end{array} \right. \]
\end{example}


\begin{theorem} \label{equivalence}
The notion of two coarse maps being coarsely homotopic is an equivalence relation.
\end{theorem}

Before proving this theorem we need a technical lemma.

\begin{lemma} \label{split_cylinder}
Let $q,p\colon X\rightarrow \R_+$ be coarse maps.  Let us write $I_{p+q}X = A\cup B$ where
\[ A =\{ (x,t)\in I_{p+1}X \ |\ t\leq p(x) \}; \qquad B =\{ (x,t)\in I_{p+q}X \ |\ t\geq p(x) \} . \]

Suppose that $f\colon I_{p+q}X\rightarrow Y$ is a map such that the
restrictions $f|_A$ and $f|_B$ are coarse maps.  Then the map $f$ is a coarse
map.
\end{lemma}

\begin{proof}
It is clear that the inverse image under the map $f$ of a bounded set is
bounded, as the union of any two bounded sets is bounded.  Let $M\subseteq (X\times \R_+)\times (X\times \R_+)$ be an entourage.  We need to show that the image $f[M]$ is an entourage.

Since the restrictions $f|_A$ and $f|_B$ are coarse, we know that the sets $f[M\cap (A\times A)]$
and $f[M\cap (B\times B)]$ are entourages.  We need to prove that
the sets $f[M\cap (A\times B)]$ and $f[M\cap (B\times A)]$ are
entourages. We will check only the first case; the second case is
similar.

Without loss of generality, suppose that $M=M_1\times M_2\cap I_{p+q}X$ where
$M_1\subseteq X\times X$ and $M_2\subseteq \R_+ \times \R_+$ are symmetric
entourages containing the diagonal, and with  $M_2 = Z(M_2)$. {We are here indulging in some mild abuse of notation involving the order of various factors in products.}  Consider points $(x,s)\in A$ and $(y,t)\in B$ such that $((x,s),(y,t))\in M$.

The inequalities $s\leq p(x)$ and $p(y)\le t$ hold.  So either $s\leq p(y)\le t$ or $p(y)\le
s\le p(x)$. The former yields that $(p(y),t) ,\ (p(y),s) \in Z(M_2)=M_2$; the latter that $(p(y),s)\in Z(p[M_1])$. Since $(s,t)\in M_2$, in either case we have that $(p(y),t) ,(p(y),s)\in Z(p[M_1])M_2$.  So if we let $N$ be the entourage
$M_1\times Z(p[M_1])M_2$ (which depends only on the entourage $M$
and the coarse map $p$), then $((x,s),(y,p(y)))\in N\cap (A\times A)$
and $((y,p(y)),(y,t))\in N\cap (B\times B)$.

Therefore
\[ (f(x,s),f(y,t))\in f[N\cap (A\times A)]f[N\cap (B\times B)]  . \]
Hence the image $f[M\cap (A\times B)]$ is contained in the entourage $f[N\cap
(A\times A)]f[N\cap (B\times B)]$ and the map $f$ is coarse.
\end{proof}

\noindent
{\bf Proof of Theorem \ref{equivalence}:}
The relation is reflexive by Example \ref{close}.  Let $p\colon X\rightarrow \R_+$ be a coarse map, and  let $H\colon I_pX\rightarrow Y$ be a coarse homotopy.  Define a map
$\overline{H}\colon I_pX\rightarrow Y$ by the formula
\[ \overline{H}(x,t) = H(x,p(x)+1-t) . \]

We claim that the map $\overline{H}$ is a coarse homotopy, thus proving that the relation of coarse homotopy is symmetric.  To show this fact, it suffices to show that the \emph{flip map} $F\colon I_pX\rightarrow I_p X$ defined by the formula $F(x,t) = (x,p(x)+1-t)$ is coarse.

Let $M\subseteq X\times X$ and $N\subseteq \R_+ \times \R_+$ be entourages.  Observe that
\[ F(M\times N) \subseteq M\times (p(M)+1-N) \]
which is an entourage by Lemma \ref{GR} as $p$ is a coarse map.

Let $A\subseteq X$ and $B\subseteq \R_+$ be bounded sets.  Then
\[ F^{-1} [A\times B] \subseteq A\times (p(A)+1-B) \]
which is bounded since $p$ is coarse, and so takes bounded sets to bounded sets.  We conclude that the map $F$ and hence the map $\overline{H}$ are coarse.

We must now prove that the equivalence relation is transitive.  Let $p,p'\colon X\rightarrow \R_+$ be coarse maps.  Then by Proposition \ref{sumax}, the sum $p+p'+1\colon X\rightarrow \R_+$ is also coarse.

Consider coarse homotopies $H\colon I_pX\rightarrow Y$ and $H'\colon I_{p'}X\rightarrow Y$ such that
$H(x,p(x)+1) = H'(x,0)$ for all $x\in X$.  Define a map
$H+H' \colon I_{p+p'+1}X\rightarrow Y$ by the formula
$$(H+H')(x,t) = \left\{ \begin{array}{ll}
H(x,t); & 0\leq t\leq p(x)+1 \\
H'(x,t-(p(x)+1)); & p(x)+1\leq t\leq p(x)+p'(x)+2 \\
\end{array} \right.$$

Then the map $H+H'$ is a coarse map by Lemma \ref{split_cylinder}.  Transitivity now follows.
\noproof

\mbox{}

The above notion of coarse homotopy is not quite the one used in older
literature for the coarse category.   However, as mentioned in \cite{Bart},
the conventional definition is not quite adequate for the purposes of coarse
homology. Our definition is the appropriate remedy.

The definition of coarse homotopy contains the choice of the basepoint
projection map $p\colon X\to\reals_+$. It might seem that we lose too much
control here. However, for most spaces we are interested in we can normalize this:
\begin{lemma}\label{lem:normalize_homotopy}
  Let $X$ be a path-metric space, considered as
  coarse space. For $x_0\in X$, let $p_0\colon X\to \reals_+; x\mapsto
  d(x,x_0)$ be the standard basepoint projection. Let $q\colon X\to \reals_+$
  be any coarse map. Then any coarse homotopy $H\colon I_qX\to Y$ between
  $f\colon X\to Y$ and $g\colon X\to Y$ gives rise to a coarse homotopy $\bar
  H\colon I_{p_0}X\to Y$ between $f$ and $g$.

  The statement generalizes in the obvious way to $X$ with finitely many path
  components. 
\end{lemma}
\begin{proof}
  As $X$ is a path metric space, it is well known that the coarse map $q$ is
  \emph{large scale Lipschitz}, i.e.~there is $L>0$ such that $|q(x)-q(y)|\le
  L d(x,y)+L$ for all $x,y\in X$. In particular, $|q(x)| \le |q(x)-q(x_0)| +
  |q(x_0)| \le L p_0(x) + C$ for $C=L+|q(x_0)|$ and for all $x\in X$. Set
  $q'(x):= C+L p_0(x)$. We just saw that $q\le q'$. We can extend the homotopy
  $H$ to $H'\colon I_{q'}X\to Y$ by extending ``constantly'' for the
  additional time, i.e.~$H'(x,t)=g(x,t)$ if $(x,t)\in I_{q'}X\setminus I_qX$.

  Finally, there is a canonical coarse equivalence $\Psi\colon I_{p_0}X\to
  I_{q'}X$, with
  \begin{equation*}
    \Psi(x,t):=
    \begin{cases}
      (x,(C+1)t); & 0\le t\le 1\\
      (x,C+1 + L(t-1)); & 1\le t \le p_0(x)+1
    \end{cases}
  \end{equation*}
  and we define $\bar H:= H'\circ \Psi$ which has all the desired properties.
\end{proof}

The following example can be found in several places in the literature, for
example following \cite[Lemma 9.9]{Roe1}.  We write out the argument again here in order to establish that everything is in order when we use our notion of coarse homotopy.

\begin{example} \label{hyperbolic}
Let $M$ be a complete simply-connected Riemannian manifold of non-positive sectional curvature.  The metric turns the manifold $M$ into a coarse space.  The exponential map $\exp \colon \reals^n \rightarrow M$ is a distance-increasing diffeomorphism.  The inverse $\log \colon M\rightarrow \reals^n$ is therefore a coarse map.

We claim that the map $\log$ is a coarse homotopy equivalence.  The problem is that the inverse map $\exp$ is not coarse; otherwise, the result would be trivial.

Let us call a map $s\colon \reals^n \rightarrow \reals^n$ a {\em radial
  shrinking} if it takes the form $s(r,\theta ) = (f(r),\theta )$ in polar
coordinates, where the map $f\colon \reals_+ \rightarrow \reals_+$ is a
distance-decreasing differentiable map with positive derivative.  Then it is
clear that any radial shrinking is coarsely homotopic to the identity
map. Moreover, it is not hard to see that also $\exp\circ s\circ \log\colon
M\to M$ is a coarse map coarsely homotopic to the identity.

Now, we can find a radial shrinking $s$ such that the composite $\exp \circ s$
is a coarse map.  By the above remark, the composites $\log \circ \exp \circ
s$ and $\exp \circ s \circ \log$ are coarsely -homotopic to identity maps, and
so the map $\log$ is a coarse homotopy equivalence as claimed.
\end{example}

In particular, Euclidean space $\reals^n$ and hyperbolic space ${\mathbb H}^n$ are coarsely homotopy equivalent.

\section{Metric Cones}
\label{sec:cones}

In this section we collect some basic properties of metric cones. In
particular, we show that for a finite simplicial complex there is a canonical
(euclidean) coarse structure (even metric structure up to bilipschitz
equivalence) on the infinite cone.

Moreover, we prove a regularity result similar to the simplicial approximation
theorem (and based on it): in our context every coarse map is coarsely
equivalent to a Lipschitz map.

\begin{definition}
Let $X$ be a subset of the unit sphere of some real Hilbert space $H$.  Then
we define the {\em metric cone with spherical base} (with the induced metric)
\[ C(X) = \{ tx \ |\ t\geq 0,\ x\in X \} . \]

  If $Y$ is a subset of some real Hilbert space $H$ we define the \emph{metric
    cone with flat base} (with the induced metric)
  \begin{equation*}
    c(Y):= \{(hx,h)\mid h\ge 0,\, x\in Y\}\subset H\times \reals.
  \end{equation*}

  For $R\ge 0$ we set $c_R(Y):=c(Y)\cap H\times [R,\infty)$, that is to say $c_R(Y)$ is the part of the
  cone of height at least $R$. If $Y$ is compact then the inclusion
  $c_R(Y)\hookrightarrow c(Y)$ is a coarse equivalence. Therefore, for us it
  usually is sufficient to consider only the part $c_R(Y)$, which is
  sometimes technically more convenient. 
\end{definition}

\begin{example}
Let $Y= S^n$ be the whole unit sphere.  Then $C(S^n)=\R^{n+1}$.
\end{example}

This definition is further reaching than it first appears.  For example, every
finite $CW$-complex is homeomophic to a subset of the unit sphere of a Hilbert
space, even of a finite dimensional one. However, it is not completely clear
whether the resulting coarse space is uniquely defined, up to coarse
equivalence, by the homeomorphism type of $X$. It is true, however, that a
finite simplicial complex gives rise to a preferred coarse type of the metric
cone (determined by the simplicial structure), what we discuss next.

\begin{lemma}
  Let $X$ be a finite connected simplicial complex. Let $f\colon X\to
  \reals^n$ and $g\colon X\to S^m$ be PL-embeddings.

  Form the cones $c(f(X))$ and $C(g(X))$. We have a canonical homeomorphism
  \begin{equation*}
  \Psi\colon c(f(X))\to C(g(X));
  (hf(x),h)\mapsto hg(x)\quad \text{for }x\in X,h\ge 0.
\end{equation*}

  If we equip each cone with either the subspace metric obtained as restriction of the
  metric on $\reals^{n+1}$ or $\reals^{m+1}$, or with the induced path metric, then the homeomorphsim $\Psi$ and the identity maps $id_{c(f(X))}$ and
  $id_{C(g(X))}$ applied when changing metrics are bilipschitz homeomorphisms.

  In particular, the bilipschitz class does not depend on the chosen
  PL-embedding, on the question whether we use a spherical base as in
  $C(g(X))$ or a euclidean base as in $c(f(X))$, nor on the question
  whether we use the induced metric from the embedding or the induced path metric.

  The same result applies to $c_R(f(X))$ for fixed $R>0$.
\end{lemma}

\begin{proof}
  It is well known that for the PL-embeddings $f$ and $g$ the subspace metric
  and the path metric on the image are bilipschitz equivalent.
  Moreover, because the maps are piecewise linear and $X$ is compact, any two
  PL-embeddings either into $\reals^n$ or into $S^m$ induce equivalent metrics
  on $X$.

  Consider now the compact cones (the parts of the full
  cones with height between $0$ and $1$) $c_f(1)$ and $C_g(1)$, where for $R>0$
  \begin{equation*}
    \begin{split}
      c_f(R) &:=\{(hx,h)\in c(f(X)) \mid 0\le h\le R\}\subset \reals^n\times [0,R]\\
      C_g(R) & :=\{tx\in C(g(X)) \mid 0\le t\le R\} \subset
      B_R(0)\subset\reals^{m+1} .
    \end{split}
  \end{equation*}
  These are again PL-embedded simplicial complexes with the resulting induced
  metrics from the embeddings, so that the identity map and the restriction of
  $\Psi$ are bilipschitz homeomorphisms for the restricted metrics and the
  induced path metrics.

  Next, observe that for arbitrary $R>0$, but fixed $f,g$ the parts $c_f(R)$ of
  the cones $c_f(X)$ and 
  $C_g(R)$ of $C_g(X)$ are just scaled versions of $c_f(1)$ and $C_g(1)$. In
  particular, the identity maps (for the path metric versus the restricted
  metric) and the map $\Psi$ (restricted to $c_f(R)$) are
  just a scaling of the corresponding maps on $c_f(1)$ and $C_g(1)$,
  respectively. This implies directly that these maps remain bilipschitz
  homeomorphisms with the same bilipschitz constant as the maps for $R=1$.

  This, in turn, implies that also the maps defined on the full cones
  are bilipschitz with the same bilipschitz constant, by the very definition
  of the Lipschitz property.
\end{proof}

Of course the spaces $c(f(X))$ and $C(g(X))$ are geodesic when equipped with path metrics.

Note that the space $c_R(f(X))$ is not bilipschitz equivalent to the full cone $c_f(X)$.

  \begin{definition}\label{def:triangulate_cone}
    Let $X\subset\reals^N$ be a finite simplicial complex simplically
    embedded.

    Write $c(X)\subset \reals^N\times [0,\infty)$ as the union of the convex
    hull of $0$ and $X\times \{1\}$, the compact cone on $X$ with the obvious
    simplicial structure and the infinitely many copies of $X\times [0,1]$
    given as $ZX(n):= \{(hx,h)\mid x\in X, h\in [n,n+1]\}$ for $n=1,2,\dots$.

    We now define a simplicial structure on $c(X)$ as follows: we use the
    $n$-th standard subdivision of $X$ on $kX\times \{k\}$ for $k\in\naturals$
    with $2^n\le k <2^{n+1}$ and the product simiplicial
    structure of Definition \ref{def:product_simpl} on $ZX(k)$ compatible with
    the so given simplicial structure on 
    the top and the bottom.
  \end{definition}

  \begin{lemma}\label{lem:bounded_geom_of_cone_triang}
    There are only finitely many strong similarity types in the simplicial
    structure of $c(X)$ given in Definition
    \ref{def:triangulate_cone}. Moreover, the lengths of the edges are
    contained in a compact interval $[a,b]$ with $0<a<b<\infty$. In
    particular, there is a positive lower bound on the width of the
    simplices and an upper bound on the diameter.
  \end{lemma}
  \begin{proof}
    Scaling does not change the strong similarity type, therefore by Lemma
    \ref{lem:Dodziuk} there are only finitely many strong similarity types
    among the simplices of the cross sections $kx\times \{k\}$ for
    $k\in\naturals$. The remaining simplices are obtained from these by two
    procedures to obtain triangulations of $X\times [0,1]$ subdividing
    $\sigma\times [0,1]$ for a simplex $\sigma$, which results in finitely
    many new strong similarity types for each similarity type of
    $\sigma$. which are then also further scaled to obtain the simplices of
    $c(X)$. Furthermore, there are finitely many more simplices at the tip of
    the cone.

    The lengths of the edges in our triangulation are bounded above because we
    perform a further standard subdivision of the cross-section (which halfs
    each original edge) as soon as the complex is scaled by $2$ in $kX\times
    \{k\}$. The standard subdivision procedure does only produce edges whose
    length is at least half the length of an edge of the original simplicial
    complex. Therefore, in the cross sections $kX\times \{k\}$ the edges are
    never shorter than the shortest edge of the original triangulation of
    $X$. The statement about the lower and upper bound on the geometry of the
    simplices of the triangulation now follows immediately.
  \end{proof}

The following proposition is needed for the technical heart of our
construction to prove the main result, contained in Section
\ref{sec:groups_of_cones}. It says 
in a very precise way that concepts of coarse maps and
coarse homotopies between cones of finite simplicial complexes can be reduced
to proper Lipschitz maps and coarse Lipschitz homotopies.

\begin{proposition}\label{prop:cone_lipschitz_appr}
  Let $X,Y\subset \reals^N$ be finite geometric simplicial complexes with
  subcomplexes $X_0\subset X$, $Y_0\subset Y$ and with cones
  $c(X)$, $c(Y)$ respectively. Then every coarse
  map of pairs $\phi\colon (c(Y),c(Y_0))\to (c(X),c(X_0))$ is close
  (i.e.~coarsely equivalent) to a proper Lipschitz 
  map of pairs $f\colon (c_L(Y),c_L(Y_0))\to (c(X),c(X_0))$ where we restrict
  the domain to the coarsely equivalent $c_L(X)$ for a 
  suitable $L>0$. The map $f$ can be chosen to be simplicial for triangulations of
  the cones as in Definition \ref{def:triangulate_cone}.

  Moreover, if the map $\phi$ is already Lipschitz when restricted to $c(Y_1)$ for a
  further subcomplex $Y_1$ of $Y$, then the maps $\phi$ and the $f$ constructed
  in the process and restricted to $c_L(Y_1)$ are Lipschitz homotopic as maps
  of pairs $(c_L(Y_1),c_L(Y_1\cap Y_0))\to c(X),c(X_0))$. Even better, the above map
 $f$ can be replaced by a coarsely equivalent
  Lipschitz map $\bar f$, which coincides with
  $\phi$ on $c(Y_1)$).

  Finally, suppose the maps $\phi,\psi\colon (c(Y),c(Y_0))\to (c(X),c(X_0))$ are equivalent by a coarse homotopy that is proper
  Lipschitz when restricted to $c(Y_1)$. Let $f$ and $g$ be proper Lipschitz
  maps constructed above coarsely equivalent to $\phi$ or $\psi$,
  respectively, with 
  $f|_{c(Y_1)}=\phi|_{c(Y_1)}$ and $g|_{c(Y_1)}=\psi|_{c(Y_1)}$.
  Then there is  a proper Lipschitz homotopy of pairs between $f$ and $g$ which coincides with
  the original homotopy on $c(Y_1)$.
 
\end{proposition}
\begin{proof}
  The strategy is to replace our map by a simplicial map for suitable and
  regular enough triangulations. The Lipschitz property will then follow from
  Lemma \ref{lem:simplicial_is_lipschitz}.

  We choose the triangulation of $c(X)$ as in Definition
  \ref{def:triangulate_cone}. 

  To obtain the desired simplicial map we follow the method of proof of the
  classical simplicial approximation
  theorem \cite[Section 3.4]{Spanier} and \cite{Zeeman}.

  For this, choose $R>0$ such that $\diam(\phi(\St(x)))\le R$ for every vertex
  $x$ 
  in $c(Y)$, where $\St(x)$ is the closed star of the vertex $x$. This is
  possible due to Lemma 
  \ref{lem:bounded_geom_of_cone_triang} (which gives a uniform upper bound on
  the diameters of all such stars) and the fact that $\phi$ is a coarse
  map between metric spaces.

  Next, consider the triangulation of Definition \ref{def:triangulate_cone} on
  $c_1(X)$. By Lemma \ref{lem:bounded_geom_of_cone_triang}, the  simplices of
  this triangulation are obtained from finitely 
  many congruence types, scaled by elements in $[a,b]$ for a compact subset of
  $(0,\infty)$. This implies that there is $r>0$ such that every $r$-ball is
  contained in the open star of a vertex (the covering by open stars of
  simplices has Lebesgue number $\ge r$). Dually, by just scaling we obtain:
  there is $L'>0$ such that for the $C$-scaled triangulation of Definition
  \ref{def:triangulate_cone} on $c_{L'}(X)$ every $R$-ball is contained in the 
  open star of a simplex.

  Use now the properness of the map $\phi$ to choose a natural number $L>0$
  such that
  $\phi(c_{L}(Y))\subset c_{L'}(X)$. Now, the standard conditions for the proof
  of the simplicial approximation theorem of \cite[Section 3.4]{Spanier}
  are satisfied: given any vertex $v$ of our triangulation of $c_{L}(Y)$, the
  images of the collection of all vertices connected to $v$ by an edge is
  contained in an open star of a vertex $w_v$ of the chosen triangulation of
  $c_{L'}(X)$. Consequently, we can now define a simplicial map $f\colon
  c_{L}(Y)\to c_{L'}(X)$ defined by sending each vertex $v$ to an appropriate
  vertex $w_v$. Automatically, as in \cite[Corollary 3.4.4]{Spanier} the subcomplex $Y_0$ will be mapped to the subcomplex $X_0$ by
  this construction. Moreover, $f$ and $\phi$ have distance at most $D$, where
  $D$ is an upper bound on the diameters of the simplices of our scaled
  triangulation of $c(X)$. Up to scaling, there are only finitely many
  isometry types of simplices. By Lemma \ref{lem:simplicial_is_lipschitz} and
  Lemma \ref{lem:Lip_extend}, the map $f$ is globally Lipschitz.

  The standard construction of simplicial approximation provides a well defined
  ``straight line'' homotopy of pairs $H\colon c_L(Y)\times [0,1]\to
  c_{L'}(X)$ between 
  $\phi$ and $f$, with $H(x,t)= (1-t)\phi(x)+ t f(x)$ where our construction
  makes sure that this is indeed making sense and given by a path inside a
  simplex of $c_{L'}(X)$. In particular, throughout this homotopy $X_0$ is
  mapped to $Y_0$. Of course, if $\phi$ is not continuous also $H$ is
  not. However, if $\phi$ is Lipschitz then the triangle inequality implies
  that $H$ is also Lipschitz, with Lipschitz constant determined by the
  maximum $K$ of the Lipschitz constants of $f$ and $\phi$ and by
  $D$. Concretely, if $x,y$ lie 
  in the same simplex of $Y$ (which suffices to consider) and $0\le s,t\le 1$
  then 
  \begin{equation*}
    \begin{split}
      |H(x,t)-H(y,s)| &= |(1-t)\phi(x)+tf(x) -(1-s)\phi(y)-sf(y)|\\
       & \le (1-t) |\phi(x)-\phi(y)| + t |f(x)-f(y)| + |t-s|\cdot
       |\phi(y)-f(y) |\\
      &\le K d(x,y) + |t-s|\cdot D
    \end{split}
  \end{equation*}
  which implies the claim. The same argument applies when we restrict
  everything to a subcomplex $c_L(Y_1)$ on which $\phi$ is Lipschitz.

  We now use this Lipschitz homotopy $H$ to change $f$ to coincide with $\phi$
  on the subcomplex $c_L(Y_1)$. For this, we use a geometric topological
  implementation of the fact that the inclusion of $c_L(Y_1)$ into $c_L(Y)$ is
  a cofibration. More specifically, consider the space $Z:=c_L(Y_1)\times
  [0,1]\cup_{c_L(Y_1)} c_L(Y)$ where we use the embedding $c_L(Y_1)\to
  c_L(Y_1)\times [0,1]; y\mapsto (y,1)$ to glue.

  We now construct a map $R\colon c_L(Y)\to Z$ which maps $c_L(Y_1)$ to
  $c_L(Y_1)\times \{0\}$ in the obvious way and which is the identity on all
  simplices of $c_L(Y)$ not touching $c_L(Y_1)$.

  Such a map is constructed by ``stretching out'' a simplex $c\sigma$ of
  $c_L(Y)$ with a face $\tau:=\sigma\cap c_L(Y_1)$ not equal to $\sigma$ to
  $\tau\times [0,1]\cup_{\tau\times\{1\}}\sigma$, i.e.~by choosing
  (compatible with face restrictions) suitable maps
  \begin{equation*}
    R_\sigma\colon \sigma\to \tau\times[0,1]\cup_{\tau\times\{1\}}\sigma
  \end{equation*}
  sending the face $\tau$ identically to $\tau\times\{0\}$ and the
  complementary face $\tau^\perp$ (spanned by all simplices of
  $\sigma\setminus\tau$) identically to $\tau^\perp$. It is an elementary
  observation that this can be done, and that this can be done such that
  restricted to each simplex the map is Lipschitz (albeit not affine
  linear). But now, because up to scaling we have only finitely many
  configurations due to Lemma \ref{lem:bounded_geom_of_cone_triang}, it
  suffices to use
  finitely maps $R_\sigma$ up to scaling to construct the map $R$. This implies
  that $R$ is globally Lipschitz.

  The map $\bar f$ is now defined as the composition of $R$ with the union of
  $f$ on $c_L(Y)\subset Z$ and the homotopy $H$ on $c_L(Y_1)\times
  [0,1]\subset Z$ which as a composition of unions of Lipschitz maps is still
  Lipschitz and also clearly coarsely equivalent to $f$.

  The statement about homotopies follows from the general (relative) statement
  applied to the coarse homotopy which by Lemma \ref{lem:normalize_homotopy}
  we can assume to be defined on $I_pc(Y)$ for $p\colon c(Y)\to \reals_+;
  (hy,h)\mapsto y$ the
  standard height projection. But, then $I_pc(Y)=c(Y\times [0,1])$, so that
  indeed we are in the situation already discussed.
\end{proof}


\begin{definition}\label{def:radial_map}

For $X,Y\subset\reals^n$ we form the cones $c(X)\subset X\times
[0,\infty)$ and $c(Y)\subset Y\times [0,\infty)$. A map $f\colon X\to Y$
induces a \emph{radial} map $c(f)\colon c(X)\to f(Y); (hx,h)\mapsto (hf(x),h)$.  
\end{definition}

\begin{remark}
  Similarly, for cones with spherical base one defines the radial map $f_C$
  induced by a 
  map $f$ between the bases of the cones.   Unfortunately, the maps $f_C$ and
  $c(f)$ are not in general coarse. They are, however, if the initial map is
  Lipschitz. 
\end{remark}

\begin{proposition}\label{prop:cone_map_Lipschitz}
Let $X$ and $Y$ be bounded subsets of Hilbert spaces (with diameter bounded by
$D$).  Let
$f\colon X\rightarrow Y$ be a proper Lipschitz map.  Then the induced map
$c(f)$ is a proper Lipschitz map. In particular, the map $c(f)$ is coarse.
\end{proposition}

\begin{proof}
If $B\subseteq c(Y)$ is compact, then the inverse image
$c(f)^{-1}[B]\subseteq c(X)$ is also compact by the properness of $f$.

Let $L$
be the Lipschitz constant of $f$. 
Let $R>0$, $s,t\in \R_+$ and $x,y\in X$, and suppose that $\| (sx,s) - (ty,t)
\| < R$.  Then it follows that $|s-t|<R$ and by the triangle inequality
\begin{equation*}
  s \|x-y\| \le \|sx-ty\| + |s-t|\|y\| \le R + RD.
\end{equation*}

Now
\[  \begin{array}{rcl}
\| c(f)(sx,x ) - c(f)(ty,t) \| & = &  \| (s f(x) - t f(y),s-t) \| \\ 
& \leq & 2(s \| f(x)-f(y)\| + |s-t| \| f(y) \| +|s-t|)  \\
                             & \leq & 2(Ls  \| x-y \| + (D+1)|s-t|) . \\
      &\leq & 2(L+1)(D+1)\cdot R
\end{array}\]
so the map $c(f)$ is Lipschitz with Lipschitz constant $\le 2(L+1)(D+1)$.
\end{proof}

Furthermore, the condition that the map $f\colon X\to Y$ is Lipschitz is not a
severe one up to homotopy, as the next result shows.

\begin{lemma}\label{lem:Lipschitz_approxi}
Let $(X,X_0)$ and $(Y,Y_0)$ be pairs of finite simplicial complexes, equipped
with simplicial metrics.  Let $f\colon (X,X_0) \rightarrow (Y,Y_0)$ be a
continuous map.  Then $f$ is homotopic to a Lipschitz map. 

Further, if $f_0,f_1\colon X\rightarrow Y$ are homotopic maps, and
$g_0,g_1\colon X\rightarrow Y$ are Lipschitz maps homotopic to $f_0$ and $f_1$
respectively, then we have a Lipschitz map $H\colon X\times [0,1]\rightarrow
Y$ such that $H(-,0)=g_0$ and $H(-,1)=g_1$.

If the map $f$ or the homotopy $H$ is already simplicial (and hence Lipschitz)
when restricted to a subcomplex $A\subset X$ then we can choose the Lipschitz
map and Lipschitz homotopy relative to $A$ (i.e.~ restricted to $A$ all maps
and homotopies coincides with the given ones).
\end{lemma}
\begin{proof}
  By the relative simplicial apprixomation theorem \cite{Zeeman}, after a
  suitable subdivision of the 
  simplicial structure of $X$ the map $f$ has a simplicial approximation $g$,
  which is homotopic to $f$, kept unchanged on the subcomplex $A$ where it
  already was simplicial and still maps $X_0$ to $Y_0$

  Restricted to each simplex with any chosen simplicial metric, the map $g$ is
  Lipschitz,
  being 
  affine linear between this simplex and a simplex of $Y$. The associated path
  metric is geodesic (by compactness of $X$). Because there are only finitely
  many simplices involved, the map $g$ is globally Lipschitz by Lemma
  \ref{lem:Lip_extend}.

  Any two metrics we obtain by subdivision and the compatible choice of a
  simplicial metric on each simplex are bilipschitz equivalent.

  The homotopy statement follows in the same way applying the relative
  simplicial approximation theorem to $X\times [0,1]$.
\end{proof}

\section{Coarse Homotopy Groups}

In order to define coarse homotopy groups, we need a coarse analogue of a basepoint in topology.

\begin{definition}
Let $X$ be a coarse space.  An {\em $\R_+$-basepoint} for $X$ is a coarse map
$i_0\colon \R_+ \rightarrow X$.

If $Y$ is another coarse space with $\reals_+$-basepoint $j_0$, then a coarse
map $f\colon X\rightarrow Y$ is termed {\em $\reals_+$-pointed} if $j_0 = f\circ i_0$. 
\end{definition}


The above definition immediately suggests the following.

\begin{definition}
Let $X$ be a coarse space.  Then we define the $0$-th {\em coarse homotopy
  set}, $\pi_0^\mathrm{coarse} (X)$, to be the set of coarse homotopy classes
of maps from $\R_+$ to $X$.
\end{definition}

For convenience, we write $[i]\in \pi_0^\mathrm{coarse} (X)$ to denote the
coarse $\R_+$-homotopy class of a map $i\colon \R_+\rightarrow X$.

\begin{example}
Let $B$ be a bounded coarse space.  Then there are no coarse maps $\R_+ \rightarrow B$, and so $\pi_0^\mathrm{coarse} (B) = \emptyset$.
\end{example}

\begin{remark}\label{rem:problems}
  Computing this coarse homotopy set is more difficult than it might seem at
  first 
  glance. The idea is of course that one counts the ``components at
  infinity''.

  In particular, one would expect $\pi^{\coarse}_0(\reals^n)$ to have
  two elements if $n=1$ and exactly one element if $n\ge 2$.

  However, we can define many coarse maps $\reals_+\to
\reals^2$, for example an embedding as a ray (a radial map, and it is easy to
see that these are all coarsely homotopic to each other), but also an
embedding which slowly spirals around the origin and (to be a proper map) out
to infinity. It is far from obvious how to homotop such a map to the radial
inclusion.

It is a consequence of the main result, Theorem \ref{theo:main}, of this paper
that the above statements are true.
\end{remark}






A {\em coarse pair} is a pair of coarse space $(X,A)$ along with a coarse map $k_A\colon A\to X$.

\begin{definition}
Let $(X,A)$ and $(Y,B)$ be coarse pairs.  A {\em coarse map} $f\colon (X,A)\rightarrow (Y,B)$ is a commutative diagram
\begin{equation*}
  \begin{CD}
    X @>f>> Y\\
    @AA{k_A}A @AA{k_B}A\\
    A @>{f}>> B
  \end{CD}
\end{equation*}
\end{definition}

\begin{definition}
Let $f,g\colon (X,A)\rightarrow (Y,B)$ be coarse maps such that $f|_A = g|_A$.
A {\em relative coarse homotopy} between $f$ and $g$ is a coarse homotopy
$H\colon I_p X\rightarrow Y$ between the maps $f,g\colon X\rightarrow Y$ such
that $H(a,t) = f (a)$ for all $a\in A$ and $t\leq p(a)+1$.
\end{definition}

We call a coarse map of pairs $f\colon (X,A)\rightarrow (Y,B)$ a {\em relative
  coarse homotopy equivalence} if there is a coarse map of pairs $g\colon
(Y,B)\rightarrow (X,A)$ such that the composites $g\circ f$ and $f\circ g$ are
relatively coarsely  homotopic to the identities $1_X$ and $1_Y$, respectively.

The following definition is directly inspired by the classical definition of
homotopy groups.
\begin{definition}
  Let $X$ be a coarse space with $\reals_+$-basepoint
  $i_0\colon \reals_+\to X$. For $n\ge 1$ define the $n$-th coarse homotopy
  group $\pi_n^{coarse} (X,i_0)$ to be the set of relative $\reals_+$-pointed coarse homotopy classes of maps \[ F \colon (c([0,1]^n),c(\boundary[0,1]^n))\to (X,i_0[\reals_+]) \] such that $F|_{c(\boundary[0,1]^n)} = i_0\circ p$.

  Here $p\colon c([0,1]^n)\to \reals_+; (x,h)\mapsto h$ just denotes the height
  variable of the cone. A homotopy is $\reals_+$-pointed if it preserves the
  $\reals_+$-basepoint throughout.

  More generally, for a coarse pair $k_A\colon A\to X$ with
  $\reals_+$-basepoint $i_0\colon \reals_+\to A$ we define the relative $n$-th
  coarse homotopy ``group'' $\pi_n^{\coarse}(X,A,i_0)$ to be the set of relative $\reals_+$-pointed coarse
 homotopy classes of maps \[ F \colon (c([0,1]^n),c(\boundary[0,1]^n),c(\boundary_+[0,1]^n))\to (X,A,i_0[\reals_+]) \]
such that $F|_{c(\boundary_+[0,1]^n)} = i_0\circ p$.

  Here $\boundary_+[0,1]^n:=\{(x_1,\dots,x_n)\in\boundary[0,1]^n\mid
  x_n>0\}$.
\end{definition}



The following result is routine to check; the computations almost identically
resemble those needed to check the corresponding in topology.  For details,
see for example \cite[Section 7.2]{Spanier}. The main points to care about are the following:
\begin{itemize}
\item The piecewise defined coarse maps indeed are globally coarse maps, this
  follows immediately from Proposition \ref{prop:coarse_glue}.
\item The usual homotopies can be used to define coarse homotopies on
  appropriate cylinders. This again works nicely and automatically, with
  cylinder $I_pc([0,1]^n)$ where $p\colon c([0,1]^n)\to \reals_+$ is again the
  height projection.
\end{itemize}

\begin{proposition}
Let $n\geq 1$.  Let $F,G \colon (c([0,1]^n),c(\boundary[0,1]^n))\rightarrow (X,i_o[\R_+] )$ be
such that $F|_{c(\partial[0,1]^n)} = G|_{c(\partial[0,1]^n)} = i_0 \circ p$.
Define there product
\begin{equation*}
F\ast G \colon (c([0,1]^n), c(\partial[0,1]^n)) \rightarrow (X,i_0[\R_+])
\end{equation*}
 by the formula
\[ F\ast G(x_1,x_2 \ldots ,x_n,h ) = \left\{ \begin{array}{ll}
F(2x_1,x_2,\ldots ,x_n,h ); & x_1\le h/2 \\
G(2x_1-h ,x_2,\ldots ,x_n,h ); &  h/2\le x_1\le h \\
\end{array} \right. \]

Then the operation $[F]\cdot [G] = [F\ast G]$ turns the set
$\pi_n^\mathrm{coarse} (X,i_0)$ into a group.  Further, $\pi_n^\mathrm{coarse}
(X,i_0)$ is abelian if $n\geq 2$. The unit is represented by the map $i_X\circ
p\colon c([0,1]^n)\to X$.

For $n\ge 2$, the same formula makes sense for the relative homotopy groups
and defines a group structure on them, abelian if $n\ge 3$.
\noproof
\end{proposition}

We call the groups $\pi_n^\mathrm{coarse} (X,A,i_0)$ the {\em coarse homotopy
  groups} 
of $(X,A)$.  The following result is also straightforward to prove, and
resembles its classical analogue. 

\begin{proposition}
Let $(X,A)$ and $(Y,B)$ be $\reals_+$-pointed coarse pairs and
$f\colon
(X,A)\rightarrow (Y,B)$ be
an $\reals_+$-pointed coarse map.  Then there is a functorially induced
homomorphism
$$f_\ast \colon \pi_n^\mathrm{coarse} (X,A,i_0)\rightarrow \pi_n^\mathrm{coarse} (Y,B,j_0)$$
defined by the formula $f_\ast ([F]) = [f\circ F]$. 

Further, if $\reals_+$-pointed coarse maps $f,g\colon (X,A,i_0)\rightarrow (Y,B,j_0)$ are
$\reals_+$-pointed relatively coarsely homotopic, then the homomorphisms
$f_\ast$ and $g_\ast$ are equal. 
\noproof
\end{proposition}

\begin{proposition}
  If $(X,A,i_0)$ is a $\reals_+$-pointed coarse pair with map $k\colon A\to X$, the analogue of the usual
  construction in topology defines a 
  long exact sequence of coarse homotopy groups or pointed sets
  \begin{multline*}
  \to \pi^{\coarse}_2(A,i_0)\xrightarrow{k_*}  \to\pi_2^{\coarse}(X,i_0)\to \pi_2^{\coarse}(X,A,i_0)
  \xrightarrow{\boundary} \pi^{\coarse}_1(A,i_0) \\
  \xrightarrow{k_*}
    \pi^{\coarse}_1(X,i_0) \to  \pi^{\coarse}_1(X,A,i_0)\to
    \pi_0^{\coarse}(A,i_0)\to\pi_0^{\coarse}(X,i_0)
  \end{multline*}
 Here, the boundary map is (as usual) obtained by restricting to the subset of
 $c[0,1]^n$ with $x_n=0$.
\end{proposition}
\begin{proof}
  The proof just follows the standard pattern of the corresponding statement
  for ordinary homotopy groups, compare \cite[Section 7.2]{Spanier}. There
  is one subtlety
  though: One has to convert certain homotopies $H\colon I_p c[0,1]^n\to
  X$ to maps $\bar H\colon c[0,1]^{n+1}\to X$.

  Usually, this is done by interpreting the homotopy parameter $t$ of
  $(hx,h,t)$ as the extra variable $hx_{n+1}$. This is permitted here, as
  well, as we can normalized the domain of the homotopies defined on
  $c[0,1]^n$ to be defined on $I_pc([0,1]^n)$ due to Lemma
  \ref{lem:normalize_homotopy} where $p(hx,h) =h$ is the standard height
  projection. 

  We leave the details to the reader.
\end{proof}

\begin{remark}
  In classical topology, probably the most important application of the long
  exact sequence of homotopy groups of a pair is to a fibration $F\to E\to B$,
  where one proceeds to identify the (in general mysterious) relative homotopy
  groups of $(E,F)$ with those of the base $B$.

  A question for the future is whether there is a version of a coarse
  fibration which is as frequent as the
  fibrations in classical topology, and for which a corresponding statement
  holds for coarse homotopy groups.
\end{remark}

\section{Homotopy groups of cones}
\label{sec:groups_of_cones}




For $X\subset \reals^n$ with basepoint $x_0\in X$ we have the corresponding
$\R_+$-basepoint $i_0 \colon \R_+ \rightarrow c(X)$, the ray through $x_0$.
The main result of this section 
is that for a wide class of spaces, $X$, the coarse homotopy group,
$\pi_n^\mathrm{coarse} (c(X),i_0)$ is isomorphic to the ordinary homotopy group
$\pi_n (X,x_0)$.  In particular, we have isomorphisms $\pi_n^{\coarse}
(\R^{k+1},i_0 ) \cong \pi_n^{\coarse}(c(S^k),i_0)\cong \pi_n (S^k,x_0)$.

  At first glance this result seems expected. At second glance, however, one
realizes that this is not such a triviality, as already discussed in Remark
\ref{rem:problems} concerning $\pi^{\coarse}_0(\reals^n)$ which of course
persists to higher degrees.

 If $X$ is a finite simplicial complex there is a canonical bilipschitz class
 of metric cones $c(X)$ coming from a PL-embedding of $X$ into Euclidean
 space, as discussed in Section \ref{sec:cones}.  If $x_0\in X$ is a basepoint, then the cone $c(X)$ has
  an induced $\R_+$-basepoint $i_0\colon \R_+ \rightarrow c(X)$ defined by
  the formula $i_0(t)=(tx_0,t)$.

\begin{definition}\label{def:Psi}
 We define the homomorphism
  \begin{equation*}
  \Psi \colon \pi_n (X,x_0) \rightarrow \pi_n^{\coarse} (c(X),i_0)
\end{equation*}
by setting $\Psi ([f]) = [c(f)]$ where $f\colon ([0,1]^n,\boundary [0,1]^n)
\rightarrow (X,x_0)$ is
  a Lipschitz map.  The equivalence class on the left is that of
  relative Lipschitz homotopy, and that on the right is relative coarse
  $\R_+$-homotopy. 
\end{definition}

Note that it follows from Lemma
  \ref{lem:Lipschitz_approxi}
  that the set of continuous homotopy classes of continuous maps here is the
  same as the set of Lipschitz homotopy classes of Lipschitz maps, so that
  the map $\Psi$ is well defined. By the construction of the group
  structures, it is a group homomorphism.

The main result in this article is that the map $\Psi$ is an isomorphism.  We prove this by constructing an inverse.  

The following result is the technical heart of our construction. It says
that we can homotop to radial Lipschitz maps. In the statement of the result and the proof, we will write points in $c_L(X)$ as pairs $(hx,h)$
with $x\in X$, $h\ge L$. 

\begin{proposition} \label{step3}
Let $X$ and $Y$ be finite simplicial complexes PL-embedded into $\reals^n$ with
subcomplexes $X_0\subset X$ and $Y_0\subset Y$, respectively. Let
$f\colon (c(X),c(X_0))\to (c(Y),c(Y_0))$ be a coarse map.
  Then, if we restrict the map $f$ to $c_L (X)$ for some suitable $L>0$, it is coarsely homotopic, as a map of pairs, to a radial proper Lipschitz map $g$.

 Suppose $X_1\subset X$ is a subcomplex such that the 
restriction of $f$ to $c(X_1)$ is already a radial Lipschitz map. Then we can chose $g$ such that
$g|_{c(X_1)}=g|_{c(X_1)}$, and the coarse homotopy between $f$ and $g$ can be chosen such that its restriction to  $c(X_1)$ is the concatenation of a homotopy of the form
  $(hx,h,t)\mapsto \rho(h,t)f(x,L)$ with its inverse, where $0\le t\le h$.
\end{proposition}

\begin{proof}
By Proposition \ref{prop:cone_lipschitz_appr}, we can assume that $f$ is a proper
Lipschitz map on $c_L(X)$ with values in $c_{L'}(Y)$ for some $L,L'\ge 1$. To
simplify notation, we assume 
$L=L'=1$, the general case is just a technical modification.

We construct our homotopy in several steps.



First, define
\[ g(hx,h) :=
f\left( \sqrt{h} x,\sqrt{h}\right).
\]

We define a proper Lipschitz homotopy $F$ between $g$ and $f$ by the
formula 
\[ F(hx,h,t) := 
f \left( (h-t) x , h-t\right); \qquad  0\le t\leq h-\sqrt{h}.
\]

Secondly, define
\[ u(hx,h) :=
\sqrt{h} f \left( {x},1 \right).
\]

We define a proper Lipschitz homotopy between $u$ and $g$ by
\begin{equation*}
  G(hx,h,t) := \left(\frac{\sqrt{h}}{\frac{t}{\sqrt{h}}+1}\right) f\left(
    (\frac{t}{\sqrt{h}}+1)x,\frac{t}{\sqrt{h}}+1\right); \qquad 0\le t\le
  h-\sqrt{h} .
\end{equation*}

Finally, let
\[ v(hx,h) :=
h f\left( {x},1 \right). \]

Then $v$ is a radial proper Lipschitz map on $c_1(X)$.  We define the proper
Lipschitz homotopy $H$ between $u$ and $v$ by
\begin{equation*}
  H(hx,h,t):= (t+\sqrt{h}) f(x,1); \qquad 0\le t\le h-\sqrt{h}.
\end{equation*}


We have given explicit formulas for the maps and the
  homotopies. Substituting $t=0$ or $t=h-\sqrt{h}$ into the homotopies, it is
  immediate to see that they are homotopies between the maps as claimed. We
  have to justify the following facts:
  \begin{enumerate}
  \item The maps are proper.
  \item The maps are globally Lipschitz.
  \item The homotopies are indeed coarse homotopies, i.e.~the domains are
    permitted.
  \item All maps send $c(X_0)$ to $c(Y_0)$.
  \item The restriction of the maps and homotopies to $c(X_1)$, when the
    original map $f$ is radial, have the required form.
  \end{enumerate}

  The  domain of the homotopies is contained in
  $I_{p'}c_1(X)$ with $p'\colon c_1(X)\to \mathbb{R}_+: (hx,h)\mapsto
  h-\sqrt{h}$ which is a proper
  Lipschitz map and therefore a coarse map.

  By construction, all maps constructed send $c(X_0)$ to $c(Y_0)$. If $f$ is
  radial, i.e.~$f(hx,h)=hf(x,1)$ then the first homotopy $F$ reduces to
  $F(hx,ht,t)=(h-t)f(x,1)$ (for $0\le t\le h-\sqrt{h}$), the second homotopy $G$
  is 
  constant, and the third homotopy $H$ becomes $H(hx,h,t)= (t+\sqrt{h})f(x,1)$
  for $0\le t\le h-\sqrt{h}$ which indeed is precisely the inverse of $F$.

    It remains to check that all maps defined are proper and Lipschitz, using
    that $f$ itself is proper and Lipschitz.

    The homotopy $F$ is the composition of $f$ and a map $\alpha\colon
    I_{p'}c_1(X)\to c_1(X)$ for which it is elementary to check that it is
    proper and Lipschitz.

    To check that $G$ and $H$ are globally Lipschitz is slightly more tedious,
    but again an 
    elementary exercise, using Lemma \ref{lem:Lip_on_prof} and Lemma
    \ref{lem:metric_est_in_cone}. Their properness follows from the fact that
    the norm of the values tends to infinity as $h\to \infty$.

    Let us give some of the details of the proof of the Lipschitz property of
    $G$, the most tedious to write down. Consider $h(x,1)\in c_1(X)$ and
    $t<s\le h-\sqrt{h}$. Then
    \begin{equation*}
      \begin{split}
        |G(hx,h,t)-G(hx,h,s)| \le &\left|
           \frac{\sqrt{h}}{\frac{t}{\sqrt{h}}+1} -
           \frac{\sqrt{h}}{\frac{s}{\sqrt{h}}+1} \right|
         f\left((\frac{t}{\sqrt{h}}+1) (x,1)\right)\\
         &+
         \frac{\sqrt{h}}{\frac{s}{\sqrt{h}}+1} \left|
           f\left((\frac{t}{\sqrt{h}}+1) (x,1)\right)  -
           f\left((\frac{s}{\sqrt{h}}+1) (x,1)\right)\right| \\
         \le & {h}\left|\frac{1}{t+\sqrt{h}} -\frac{1}{s+\sqrt{h}}\right|
         \left( L (\frac{t}{\sqrt{h}}+1) L\right)\\
         &+ \frac{h}{s+\sqrt{h}} L |t-s| \frac{L}{\sqrt{h}}\\
         \le & h \frac{1}{(t+\sqrt{h})^2} |t-s|
         \left(\frac{L}{\sqrt{h}}(t+\sqrt{h})L\right)\\
         &+ \frac{\sqrt{h}}{s+\sqrt{h}}L^2 |t-s|\\
         \le & 2L^2 |t-s| = 2L^2 |(hx,h,t)-(hx,h,s)|
       \end{split}
     \end{equation*}
  The first inequality is just the triangle inequality. For the second, we use
  the Lipschitz property of $f$ (with Lipschitz constant $L$) which implies in
  particular also that $|f(x)|\le L|x|$ for all $|x|\ge 1$ by
  comparing to $f(0)$ and making $L$ bigger depending on $f(0)$, if
  necessary. We also use that by the compactness of $X$ we can 
  choose $L$ such that $|(x,1)|\le L$ for all $x\in X$.

  For the third inequality we use that the derivative $y\mapsto
  -(y+\sqrt{h})^{-2}$ of $y\mapsto
  (y+\sqrt{h})^{-1}$ is monotonically increasing in absolute value and use the
  mean value theorem.

  For the last inequality, we just use that $s,t\ge 0$.

  By Lemma \ref{lem:Lip_on_prof}, the Lipschitz property of $G$ follows now if
  we establish a similar uniform inequality for $G(hx,h,t)-G(ry,r,t)$ for
  $x,y\in X$, $1\le h<r$, and $t\le h-\sqrt{h}$ which can be obtain by similar
  elementary computations, using also Lemma
  \ref{lem:metric_est_in_cone}. Details are left to the reader. 
\end{proof}

We used the following Lipschitz criterion for coarse homotopies.
\begin{lemma}\label{lem:Lip_on_prof}
  Let $X,Y$ be metric space, $p_o\colon X\to [0,\infty); x\mapsto d(x,x_0)$ a
  basepoint projection for $x_0\in X$. Let $H\colon I_{p_0}X\to Y$ be a map.

  If there is $C>0$ such that for each $t_0\in [0,\infty)$ and each $z\in X$
  the restrictions to the $t_0$-time slice $S_{t_o}:= X\times \{t_0\}\cap
  I_{p_0}X$ and the $z$-slice 
  \begin{equation*}
    H_{S_{t_0}}\colon S_{t_0}\to Y;\qquad H_{\{z\}\times [0,\infty)\cap
      I_{p_o}X}\colon \{z\}\times[0,\infty)\cap I_{p_0}X \to Y
  \end{equation*}
   are $C$-Lipschitz, then $H$ is globally $2C$-Lipschitz.
\end{lemma}
\begin{proof}
  This uses the fact that there are enough points in $I_{p_0}X$ to interpolate; specifically, let
  $(x,t)$ and $(z,s)\in I_{p_0}X$ with $p_0(x)\ge p_0(z)$. By definition of
  $I_{p_0}X$ then $s\le p_0(z)\le p_0(x)$ and therefore also $(x,s)\in
  I_{p_0}X$. Consequently, by the triangle inequality,
  \begin{equation*}
    \begin{split}
      d(f(x,t),f(y,s)) &\le d(f(x,t),f(x,s))+d(f(x,s),f(y,s)) \\
      & \le C d((x,t),(x,s)) + Cd((x,s),(y,s)) = C\abs{t-s} + C d(x,y)\\
      &\le 2C d((x,t),(y,s)).
    \end{split}
  \end{equation*}
\end{proof}

\begin{lemma}\label{lem:metric_est_in_cone}
  Assume that $X\subset\reals^N$ is bounded, i.e.~there is $C>0$ such that
  $\abs{x}\le C$ for all $x\in X$. For $(hx,h),(ry,r)\in c(X)\subset
  \reals^N\times [0,\infty)$ and $t\le \min\{h,r\}$ we then have
\[ d((tx,t),(ty,t)) \le (1+(C+1)) d((hx,y),(ry,r)). \]
\end{lemma}

\begin{proof}
Observe:
  \begin{equation*}
    \begin{split}
      d((tx,t),(ty,t)) &= td(x,y)\le d((hx,h),(hy,h))\\
      &
      \le d((hx,h),(ry,r)) + d((ry,r),(hy,h)) \\
      &= d((hx,h),(ry,r)) + \abs{r-h} \cdot \abs{(y,t)}\\
        &\le (1+(C+1)) d((hx,y),(ry,r)).
    \end{split}
  \end{equation*}
\end{proof}

Proposition \ref{step3} is not quite good enough for our purposes because the
homotopy constructed there would, for example, not preserve an
$\reals_+$-basepoint. However, we have enough control such that we can perform a
``padding'' construction in our specific situation, where the domain is
$c([0,1]^n)$ and where the map is radial on one of the faces.

\begin{lemma}\label{lem:pad_basepoint}
  Let $f\colon (c([0,1]^n),c(\partial [0,1]^n))\to (c(Y),c(Y_0))$ be a coarse
  map of coarse pairs. Set $D:=[0,1]^{n-1}\times 
  \{1\}$  and assume that $f|_{c(D)}=c(u)$ is radial for a PL-map $u$.

Then we can find a coarse homotopy of pairs from $f$ to a radial map such that
the 
restriction to $c(D)$ is equal to $f|_{c(D)}$ throughout the coarse homotopy.
\end{lemma}
\begin{proof}
  Let $i\colon [0,1]^{n-1}\times [0,1/2]\to [0,1]^n; (x_1,\dots,x_n)\to
  (x_1,\dots,x_{n-1},2x_n)$ and $p\colon [0,1]^{n-1}\times [1/2,1]\to D;
  (x_1,\dots,x_n)\mapsto (x_1,\dots,x_{n-1},1)$. Define a map
  $\bar f\colon c([0,1]^n)\to Y$ by $\bar f|_{c([0,1]^{n-1}\times [0,1/2])}:=
    f\circ c(i)$ and $\bar f|_{c([0,1]^{n-1}\times[1/2,1])}:= f\circ c(p)$,
    i.e.~we squeeze $f$ into the lower half of $c([0,1]^n)$ and then extend
    constantly in the $x_n$-coordinate.

    There is an obvious coarse homotopy between $f$ and $\bar f$ whose
    restriction to $c(D)$ is $f|_{c(D)}$ throughout the homotopy.

    Now we construct the required coarse homotopy from $\bar f$ to a radial
    map whose restriction to $c(D)$ remains constant. For this, we use the
    homotopy $H$ provided by Proposition \ref{step3} on $I_pc([0,1]^{n-1}\times
      [0,1/2])$. On the top part of the domain of this homotopy, where the
      initial map was radial, i.e.~of the 
      form $c(u)$ for a map $u\colon [0,1]^{n-1}\times \{1/2\}\to Y$, we know
      that 
      $H(h(x,t))=\rho(h,t)\cdot (u(x),1)$ with a real valued function $\rho$
      with 
      $\rho(h,t)=\rho(h,1-t)$. We then simply extend the homotopy to
      $I_pc([0,1]^{n-1}\times [1/2,1])$ by setting
      \begin{multline*}
        H(h(x_1,\dots,x_n),1,t):=
        \begin{cases}
          \rho(h,t(1-2x_n))( u(x_1,\dots,x_{n-1},1/2),1); & 0\le t\le 1/2\\
          \rho(h,(1-t)(1-2x_n))( u(x_1,\dots,x_{n-1},1/2),1); & 1/2\le t\le 1.
        \end{cases}
      \end{multline*}

      It is clear that this procedure does the job.
\end{proof}

We now formulate and prove the main result of this paper.

\begin{theorem}\label{theo:main}
Let $X$ be a finite simplical complex with subcomplex $X_0$ and base vertex
$x_0\in X_0$. Choose a
PL-embedding into $\reals^n$ and identify $X$ with its image and let
$i_0\colon [0,\infty)\to c(X); t\mapsto 
(tx_0,t)$ be associated to $x_0$.
Then the homomorphism $\Psi \colon \pi_n (X,X_0,x_0) \rightarrow \pi_n^{\coarse}
(c(X),c(X_0),i_0)$ of Definition \ref{def:Psi} is an isomorphism.
\end{theorem}

\begin{proof}

We want to construct an inverse $\Phi$ to $\Psi$. For this, let $f\colon
c([0,1]^n)\to c(X)$ be a coarse map representing an element $[f]\in
\pi_n^{\coarse}(c(X),c(X_0),i_0)$.

Observe that there is a PL-homeomorphism $[0,1]^n\to [0,1]^n$ mapping
$\partial_+[0,1]^n$ to the set $D$ of Lemma \ref{lem:pad_basepoint}. We can therefore
apply Lemma \ref{lem:pad_basepoint} and get a coarse homotopy which is
constant on $c(\partial_+([0,1]^n))$ (meaning it is an
$\reals_+$-pointed coarse homotopy) to a radial map
$c(u)$ for a PL-map $u\colon ([0,1]^n,\partial [0,1]^n,\partial_+[0,1]^n)\to
(X,X_0,x_0)$.

Of course, we want to set $\Phi([f]):= [u]$. It is then obvious that
$\Phi\circ \Psi=\id$ and $\Psi\circ\Phi=\id$.

But we have to show that the map $\Phi$ is really well defined.

For this, we could replace $f$ by $g$, coarsely homotopic through a coarse
homotopy $H_1$. Moreover, we have to chose a coarse homotopy $H_0$ from  a
radial map $c(u)$ to $f$ and $H_2$ from $g$ to a radial map $c(v)$. All
homotopies are $\reals_+$-pointed and map the boundary of $[0,1]^n$ to
$c(X_0)$. We can concatenate $H_0,H_1,H_2$ and reinterpret the domain of the
homotopy as $c([0,1]^{n+1})$. Being $\reals_+$-pointed, this concatenation is
radial when restricted to $c(\partial_+[0,l]^n\times [0,1])$.

Proposition \ref{step3} almost allows us to
replace $f$ by a coarsely equivalent radial map based on some $u\colon
[0,1]^n\to X$ such that $[f]=[c(u)]\in \pi_n^{\coarse}(c(X),c(X_0),i_0)$ such
that $[u]\in \pi_n(X,X_0,x_0)$ would be a candidate for $\Phi([f])$.

The problem is that the construction of Proposition \ref{step3} does not
preserve the coarse basepoint $i_0$. Fortunately, Proposition \ref{step3}
provides enough control on the part of the domain where the map is already
radial, in particular in our case on $c(\boundary_+[0,1]^n)$. Moreover, it is
radial on $c([0,1]^n\times\{0,1\})$ because the beginning and end of the
concatenated coarse homotopies are radial.

Now, by Lemma \ref{lem:pad_basepoint} there is a coarse homotopy to a radial map,
and that new map coincides with the old one where it is already radial. We can
reinterpret that map as (the cone of) a homotopy between $u$ and $v$ which is
pointed. This shows that indeed the map $\Phi$ is well defined.

\end{proof}

\bibliographystyle{plain}

\bibliography{data}

\end{document}